\documentclass[12pt]{scrartcl}

\usepackage{amsmath,amssymb,amsfonts}
\usepackage{latexsym}
\usepackage{epsfig,overpic}
\usepackage{array}
\usepackage{microtype}
\usepackage[colorlinks=true,linkcolor=blue,citecolor=blue,pdfborder={0 0 0}]{hyperref}
\DeclareGraphicsExtensions{.pdf,.eps,.ps,.eps.gz,.ps.gz,.eps.Y}

\usepackage{tikz}
\usetikzlibrary{shapes,arrows,snakes,calendar,matrix,backgrounds,folding,calc,positioning,patterns,spy}

\usepackage{booktabs}

\usepackage{amsthm}

\newtheorem{thm}{Lemma}
\newtheorem{lemma}[thm]{Lemma}
\newtheorem{corollary}[thm]{Corollary}
\newtheorem{remark}{Remark}

\newcommand{\vertiii}[1]{{\left\vert\kern-0.25ex\left\vert\kern-0.25ex\left\vert #1 
    \right\vert\kern-0.25ex\right\vert\kern-0.25ex\right\vert}}

\title{Removing the stabilization parameter in fitted and unfitted symmetric Nitsche formulations}

\author{Christoph Lehrenfeld
  \thanks{Institute for Computational and Applied Mathematics, e-mail: christoph.lehrenfeld@gmail.com}
}

\begin{document}
  \maketitle





\begin{abstract}
In many situations with finite element discretizations it is desirable or necessary to impose boundary or interface conditions not as essential conditions -- i.e. through the finite element space -- but through the variational formulation. One popular way to do this is Nitsche's method. 
In Nitsche's method a stabilization parameter $\lambda$ has to be chosen ``sufficiently large'' to provide a stable formulation. 
Sometimes discretizations based on a Nitsche formulation are criticized because of the need to manually choose this parameter. 
While in the discontinuous Galerkin community variants of the Nitsche method -- known as ``interior penalty'' method in the DG context -- are known which do not require such a manually chosen stabilization parameter, this has not been considered for Nitsche formulations in other contexts. 
We introduce and analyse such a parameter-free variant for two applications of Nitsche's method.
First, the classical Nitsche formulation for the imposition of boundary conditions with fitted meshes and secondly, an \emph{unfitted} finite element discretizations for the imposition of interface conditions is considered. 
The introduced variants of corresponding Nitsche formulations do not change the sparsity pattern and can easily be implemented into existing finite element codes. 
The benefit of the new formulations is the removal of the Nitsche stabilization parameter $\lambda$  while keeping the stability properties of the original formulations for a ``sufficiently large'' stabilization parameter $\lambda$. 
\end{abstract}

\section{Introduction}
Nitsche's method was originally introduced in \cite{nitsche71} and has been applied in many different contexts in order to weakly enforce interface or boundary conditions. 
The idea is to impose boundary or interface conditions not into the finite element space as essential condition but to change the variational formulation to achieve the condition at least in a weak sense. For the boundary value problem $- \Delta u = 0$ in $\Omega$ with $u|_{\partial \Omega}=0$ the classical Nitsche variational formulation is: Find $u \in V_h$ such that
$$
 \int_{\Omega} \nabla u \nabla v \,d{x}
 - \int_{\partial \Omega} \partial_n u \, v \,d{s}
 - \int_{\partial \Omega} \partial_n v \, u \,d{s}
 + \frac{\lambda}{h} \int_{\partial \Omega}  u v \,d{s} 
= 0 \quad \forall~v \in V_h,
$$
where $V_h$ is a finite element space with $V_h \subset H^1(\Omega)$ and $\partial_n v_h \in L^2(\partial \Omega)$ for all $v_h \in V_h$. 
Here $h$ is the characteristic mesh size and $\lambda$ is a parameter that has to be chosen sufficiently large.
The criticism on Nitsche's method is sometimes related to this parameter $\lambda$ and the ``sufficiently large'' condition. Although theoretical error analyses can give a precise lower bound on the choice of $\lambda$ in practice the parameter is typically chosen by estimating the order of magnitude or a trial and error approach. 
If the parameter is chosen too small, the formulation gets unstable. If however the stabilization parameter is chosen too large the condition number of the system matrix increases and error bounds degenerate. 
A suitable choice for $\lambda$ is specifically difficult if problems with a large contrast in the diffusion parameter are considered. 

In this paper we focus on \emph{unfitted} finite element discretizations with a weak enforcement of interface or boundary conditions with a \emph{symmetric} Nitsche-type formulation. In unfitted finite element methods a computational mesh is used which is not aligned to the geometries with respect to which a solution to a PDE should be approximated. We consider an elliptic interface problem with an unfitted interface. 
To transition into this problem we also consider the model problem which has been considered in the original paper \cite{nitsche71}, a Poisson problem on domain with a fitted boundary.

We give a brief overview on the literature of fitted and especially unfitted discretizations using Nitsche's method. 
Nitsche's method has a close relation to (stabilized) Lagrange multiplier methods. This connection allows to derive suitable Nitsche formulations from stabilized Lagrange multiplier formulations and has been nicely discussed in \cite{verfuerth91,stenberg95,Juntunen2015}.

While we want to focus on \emph{symmetric} Nitsche-type formulations we mention the papers \cite{freundstenberg95,burman12,Boiveau02092015} which investigate \emph{non-symmetric} Nitsche formulations and get rid of the ``sufficiently large'' condition on $\lambda$, \cite{freundstenberg95}, or the whole penalty term, \cite{burman12,Boiveau02092015}. We aim at achieving the same goal with a \emph{symmetric} formulation.

Many discontinuous Galerkin (DG) schemes -- especially the (symmetric) interior penalty method \cite{arnold1982interior} -- have a very close relation to Nitsche's method. The continuity of the solution is not build into the finite element space but has to be implemented by the variational formulation.
For DG schemes for elliptic problems many variants are known, cf. \cite{arnoldbrezzicockburnmarini02} for an overview of different methods. Among them there exist versions which get rid of the parameter $\lambda$ and are able to provide consistency, continuity, symmetry and stability. The approach discussed in this paper is in the same virtue as the DG method presented in \cite{bassirebay97,bassirebay97b} and analyzed in \cite{brezzietal99}.  

For unfitted interface problems a stable Nitsche formulation has been introduced in the original paper \cite{hansbo2002unfitted}. 
One important aspect of this method is the choice of a geometry-dependend averaging operator which is crucial for stability. 
We detail on this method below. 
Many discretization are based on this formulation, for instance \cite{hansbo05,reuskenlehrenfeld12,annavarapuetal12,hansboetal14,reuskenlehrenfeld14}. 

The major contribution of this paper is the introduction of new finite element formulations for the weak imposition of boundary and interface conditions for \emph{fitted} and \emph{unfitted} meshes. These new formulations inherit the stability from known Nitsche-type formulations but remove the stabilization parameter $\lambda$.

The paper is organised as follows. In Section \ref{sec:prelim} we introduce notation for Nitsche-type formulations and their analysis.
Based on a common framework we present Nitsche-type formulations from the literature and discuss their theoretical properties in Section \ref{sec:nitsches}. Modified formulations which get rid of the penalty parameter $\lambda$ are then introduced and discussed in Section \ref{sec:lift}.
The modified formulations involve a lifting operator. That an implementation of this lifting operator is practically feasible is shown in Section \ref{sec:impl}.

\section{Preliminaries} \label{sec:prelim}
\subsection{Problem classes}
We consider two classes of problems on which Nitsche's method is applied, cf. Fig. \ref{fig:cases}. A Boundary value problem with a boundary-\emph{fitted} mesh and an elliptic interface problem with an interface-\emph{unfitted} mesh.
\begin{figure}
\begin{center}
\begin{tabular}{c@{\qquad\qquad}c}
 \small fitted boundary value problem & \small unfitted interface problem\\ 
        \includegraphics[width=0.25\textwidth]{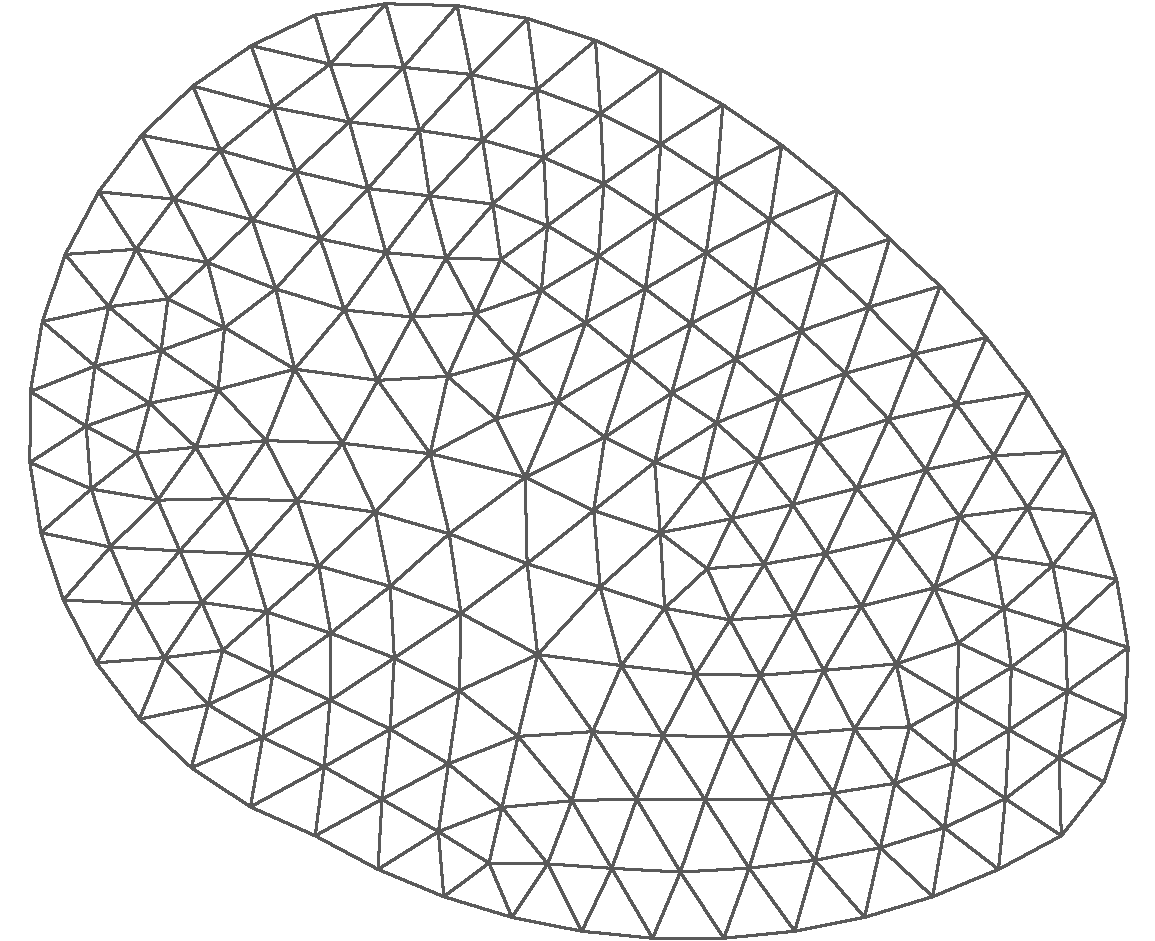}
  &
        \includegraphics[width=0.325\textwidth]{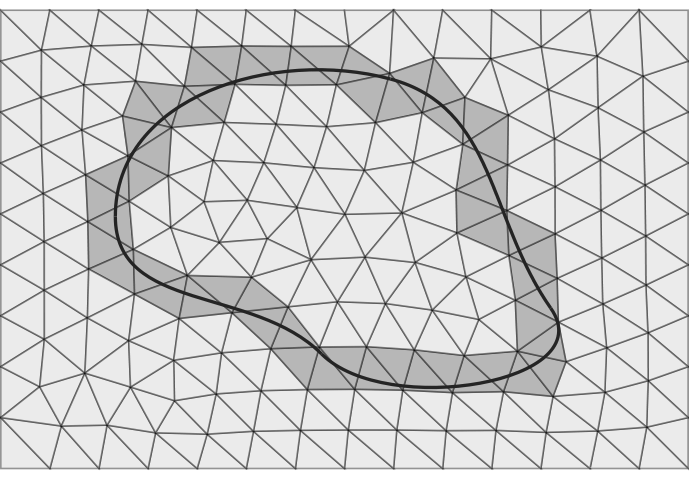}
\end{tabular}
\end{center}
\vspace*{-0.5cm}
\caption{Two situations where Nitsche's method is applied: A fitted Boundary value problems (left) and an unfitted interface problem (right).}
\label{fig:cases}
\end{figure}
In Section \ref{sec:nitsches} we introduce suitable Nitsche-type finite element formulations which are known to be stable for both problems.
To provide stability different versions of trace inverse inequalities play an important role. The stability of the classical Nitsche method for the model problem can be shown based on standard trace inverse inequalities for polynomials. 
For the interface problem with an interface-unfitted mesh a special weighting in the Nitsche method is required to obtain a similar result. 

\begin{remark}
We do not consider the case of an interface problem with a fitted mesh as the same techniques as for the Poisson problem with a boundary-fitted mesh can directly be carried over. Further the case of an unfitted boundary value problem is not consider. Existing stable Nitsche-type formulations for this problem introduce a second stabilization term with a second stabilization parameter, cf. \cite{burman2012fictitious}. With the techniques discussed in this paper it is not clear how to remove also this parameter. 
\end{remark}

\subsection{Unified notation for Nitsche formulations}
In both cases we assume to have a domain $\tilde \Omega$ on which a shape regular simplex triangulation $\mathcal{T}_h$ is available. Throughout the paper we assume quasi-uniformity with a characteristic mesh size $h$. This restriction simplifies the presentation but is not essential and can be generalized by exchanging $h$ with a localized version, $h_T$, a characteristic length of an element $T$. All scalar product's are $L^2$ scalar products, i.e. $(u,v)_S = \int_S u\ v \ ds$ for a corresponding domain $S$. 

We can write a suitable version of Nitsche's method for all considered problems in the form: Find $u \in V_h$ such that
\begin{equation} \label{eq:nitsche}
A(u,v) := a(u,v) + N^c(u,v) + N^c(v,u) + N^s(u,v) = f(v) \quad \text{ for all } v \in V_h,
\tag{N}
\end{equation}
with bilinear and linear forms which are yet to be defined. 
Here, the bilinear form $a(\cdot,\cdot)$ is responsible to provide consistency with respect to the volume terms, i.e. the inner part of the domain(s). The bilinear form $N^c(\cdot,\cdot)$ is responsible for the consistency w.r.t. boundary terms stemming from partial integration of the elliptic operator. The second term which involves $N^c(\cdot,\cdot)$ is added for symmetry (or adjoint consistency) reasons. Finally $N^s(\cdot,\cdot)$ is a term added to provide stability.  
In the standard formulation, we have
\begin{equation} \label{eq:lambda}
N^s(\cdot,\cdot) = N^s_{\lambda}(\cdot,\cdot) = \lambda \cdot N^s_{1} ( \cdot, \cdot)
\end{equation}
where $N^s_{\lambda}(\cdot,\cdot)$ is the $\lambda$-scaled penalty term. 

\subsection{Analysis of (symmetric) Nitsche formulations}
The analysis of corresponding Nitsche's methods uses standard arguments. With coercivity, continuity and consistency of the bilinear form w.r.t. suitable spaces and norms a C\'ea-like Lemma can be obtained. We briefly introduce the three criteria -- in terms of the equations \eqref{eq:coerc}-\eqref{eq:consist} below -- together with the notation for the involved norms and spaces. The criteria are fulfilled for both formulations mentioned below. 

We use the notation $a \lesssim b$ ($a \gtrsim b$) if there exists a constant $c$ independent of discretization or material parameters such that $a \leq c b$ ($a \geq c b$). If $a \lesssim b$ and $a \gtrsim b$ we write $a \simeq b$.

\paragraph{Coercivity}
The crucial component with respect to the stabilization parameter $\lambda$ is coercivity. 
For sufficiently large $\lambda$ all considered Nitsche formulations provide that the following holds:
\begin{equation}\label{eq:coerc}
A(u_h,u_h) \gtrsim \vertiii{u_h}_A^2 \text{ for all }u_h \in V_h \text{ with } \vertiii{v}_A := \sqrt{a(v,v)+N^s(v,v)}.
\tag{A}
\end{equation}
In the analysis a stronger norm is needed to bound the flux terms in the bilinear form $N^c(\cdot,\cdot)$. We denote this norm by $\vertiii{\cdot}$ and require that there holds 
\begin{equation}\label{eq:normeq}
\vertiii{v_h} \simeq \vertiii{v_h}_A \text{ for all } v_h \in V_h,
\tag{B}
\end{equation}
so that \eqref{eq:coerc} yields $A(u_h,u_h) \gtrsim \vertiii{u_h}^2$ for all $u_h \in V_h$. 
\paragraph{Continuity}
On the space $V_h + V_{\text{reg}}$ with $V_{\text{reg}}$ so that $A(u,v)$ (especially $N^c(u,v)$) is well-defined for $u,v \in V_{\text{reg}}$ we then have continuity in the sense that
\begin{equation} \label{eq:cont}
 A(u,v) \lesssim \vertiii{u} \ \vertiii{v} \quad u,v \in V_h + V_{\text{reg}}.
\tag{C}
\end{equation}
\paragraph{Consistency}
With $u \in V_{\text{reg}}$ the solution of the original problem and $u_h \in V_h$ the solution to the discrete one, we demand Galerkin orthogonality, i.e.
\begin{equation}\label{eq:consist}
A(u-u_h,v_h)=0 \text{ for all }v_h \in V_h.
\tag{D}
\end{equation}

With these consistency, continuity and coercivity results we have a version of the well-known C\'ea Lemma:
\begin{lemma}
Let $u \in V_{\text{reg}}$ be the solution to the original problem and $u_h \in V_h$ the solution to the discrete one.
If the equations \eqref{eq:coerc}-\eqref{eq:consist} are fulfilled, then there holds
\begin{equation}
\vertiii{u-u_h} \lesssim \inf_{v_h \in V_h} \vertiii{u-v_h}.
\end{equation}
\end{lemma} 

Below, in Section \ref{sec:lift} we propose an alternative to the stabilization form \eqref{eq:lambda} of the form
$$
N^s(\cdot,\cdot) = N^s_{nn}(\cdot,\cdot) + N^{s}_1(\cdot,\cdot)
$$
where $N^s_{nn}(\cdot,\cdot)$ is responsible to provide non-negativity in the sense that 
$$
  A(u,u) - N^s_1(u,u) \gtrsim a(u,u)
$$
which implies \eqref{eq:coerc} with $\lambda=1$. Then, $N^s_1(\cdot,\cdot)$ is only required to provide that $A(\cdot,\cdot)$ induces a norm. The scaling of this term is however independent of any stability requirement.
In the following section we specify the problem classes and the bilinear forms, spaces and norms corresponding to suitable discretization which are known to be stable (for sufficiently large $\lambda$).

\section{Stable Nitsche formulations} \label{sec:nitsches}
In the following subsections we will introduce suitable discretizations for the considered problems and show the conditions \eqref{eq:coerc}-\eqref{eq:consist}.
\subsection{Nitsche formulation for a boundary value problem with a body-fitted mesh} \label{sec:fittedpoisson}
In the fitted case we have $\tilde \Omega = \Omega$ and the model problem is
\begin{equation}
  -  \Delta u = f \mbox{ in } \Omega, ~ u = g~ \mbox{ on } \partial \Omega
\end{equation}
In this case the Nitsche's method, as introduced in \cite{nitsche71}, takes the form of \eqref{eq:nitsche} with 
$$
a(u,v)=(\nabla u, \nabla v)_{\Omega}, \ 
N^c(u,v) = (-  \partial_n u, v)_{\partial \Omega} \text{ and }
N^s(u,v) = N_\lambda^s(u,v) = \frac{\lambda}{h} (v, v)_{\partial \Omega}.
$$
We briefly verify \eqref{eq:coerc}-\eqref{eq:consist} with respect to suitable spaces and norms.
As the finite element space $V_h$ we consider a standard finite element space of continuous piecewise polynomials up to degree $k$. 
The stronger norm on $V_h + V_{\text{reg}}$ with $V_{\text{reg}} = H^2(\Omega)$ is
$$
\vertiii{u}^2 = \vertiii{u}^2_A + h \Vert \nabla u \Vert_{\partial\Omega}^2,
$$
and continuity on $V_h + V_{\text{reg}}$ with respect to $\vertiii{\cdot}$ follows directly with Cauchy-Schwarz, so that we have \eqref{eq:cont}. Consistency follows directly from partial integration, so that also \eqref{eq:consist} holds.
An essential estimate to prove \eqref{eq:coerc} is the following trace inverse inequality:
$$
h \Vert \nabla u \Vert_{\partial\Omega \cap T}^2 \leq h \Vert \nabla u \Vert_{\partial T}^2 \leq c_{\text{tr}} \Vert \nabla u \Vert_{T}^2, \quad u \in \mathcal{P}^k(T),~ T\in\mathcal{T}_h
$$
with a constant $c_{\text{tr}}$ only depending on the shape regularity of the element $T$ and the polynomial degree $k$, cf. \cite{warburtonhesthaven03}. Assuming that $c_{\text{tr}}$ denotes the upper bound of the shape regularity constant on all elements in the mesh, then we can easily show \eqref{eq:normeq} and the estimate 
\begin{equation} \label{eq:help}
2 N^c(u_h,u_h) \leq 2 h^{\frac12} \Vert \nabla u_h \Vert_{\partial\Omega} \cdot h^{-\frac12} \Vert u_h \Vert_{\partial\Omega} \leq \frac{8 c_{\text{tr}}}{\lambda} a(u_h,u_h) + \frac12 N^s_\lambda(u_h,u_h)
\end{equation}
from which we deduce \eqref{eq:coerc} for $\lambda$ sufficiently large ( for instance $\lambda > 16 c_{\text{tr}}$). 

\subsection{Nitsche formulation for an interface problem with an unfitted mesh} \label{sec:nuif}
For the case of an unfitted interface we assume $\tilde \Omega = \Omega$ but have an additional internal interface $\Gamma \subset \Omega$ which separates $\Omega$ into two disjoint sub-domains $\Omega_1$ and $\Omega_2$ which are unfitted, i.e. not aligned with the mesh.
Inside the domains only diffusion takes place, the diffusion coefficient $\alpha$ may however be discontinuous across the interface $\Gamma$. 
On the boundary we consider homogeneuous Dirichlet conditions to be implemented as essential conditions. This simplifies the presentation. For a weak imposition of the boundary conditions, we refer to the treatment of the case in section \ref{sec:fittedpoisson}. At the interface two conditions are posed: continuity of the solution and conservation of the flux. 
\begin{equation}
  - \alpha \Delta u = f \mbox{ in } \Omega_i, \quad u = 0~ \mbox{ on } \partial \Omega, \quad
  [\![ u ]\!] = 0 ~\mbox{ on } \Gamma, \quad
  [\![ - \alpha \partial_n u ]\!] = 0 ~\mbox{ on } \Gamma.
\end{equation}
In unfitted FEM 
a famous formulation of the problem goes back to the original paper \cite{hansbo2002unfitted}.
The finite element space depends on $\Gamma$ and is defined as $V_h = V_h^1 +V_h^2$ where $V_h^i = W_h|_{\Omega_i}$ is the fictitious domain finite element space related to domain $\Omega_i$ with the underlying standard finite element space $W_h$ which again is the space of continuous piecewise polynomial functions. The variational formulation proposed in \cite{hansbo2002unfitted} can be casted into the form of \eqref{eq:nitsche} with 
\begin{align*}
a(u,v) = (\alpha \nabla u, \nabla v)_{\Omega_1 \cup \Omega_2}, \
N^c(u,v) = ( \{\!\!\{ - \alpha \partial_n u \}\!\!\}, [\![v]\!])_{\Gamma}
\text{ and } N^s(u,v) = \frac{\lambda}{h} ( [\![u]\!],  [\![v]\!])_{\Gamma}.
\end{align*}
where $[\![\cdot]\!]$ is the usual jump operator across the interface $\Gamma$. 
The averaging operator $\{\!\!\{ v \}\!\!\} = \kappa_1 v|_{\Omega_1} + \kappa_2 v|_{\Omega_2}$ takes an important role to provide stability. 
With $V_{\text{reg}} := H^2(\Omega_1) \cup H^2(\Omega_2)$ and the norm 
$
\vertiii{u}^2 = \vertiii{u}_A^2 + h \Vert \{\!\!\{ \nabla u \}\!\!\} \Vert_{\Gamma}^2
$ 
one easily shows \eqref{eq:cont} and \eqref{eq:consist}. It remains to show \eqref{eq:coerc} and \eqref{eq:normeq}.
In the following we assume that a stable weighting is used, so that the following inverse estimate holds:
$$
\alpha_i^2 \kappa_i^2 h \Vert \nabla u \Vert_{\Gamma \cap T}^2 \leq c_{\text{tr}} \alpha_i \Vert \nabla u \Vert_{T_i}^2, \quad u \in \mathcal{P}^k(T_i),~~T_i := T \cap \Omega_i,~~ i=1,2, .
$$
for a constant $c_{\text{tr}}$ only depending on the shape regularity of $T$ (not $T_i$!), the polynomial degree $k$ and the diffusion parameter $\alpha_i$. 
We note that a carefully chosen stabilization parameter $\lambda$ in $N^s(u,v) = N_\lambda^s(u,v)$ scales with $c_{\text{tr}}$ and hence with the diffusion parameter $\alpha$.

For piecewise linears ($k=1$) and the weighting $\kappa_i = {|T_i|}/{|T|}$ the inverse inequality has been proven in \cite{hansbo2002unfitted}. For higher order discretizations we refer to the discussion of weightings that lead to stable discretizations in \cite[Remark 6]{lehrenfeld2016analysis}. Analogously to the estimates in \eqref{eq:help} this allows to prove \eqref{eq:coerc} and \eqref{eq:normeq}.
Together with the optimal approximation properties of $V_h$, see also \cite{hansbo2002unfitted}, we can obtain optimal order error bounds provided that $\lambda$ is chosen sufficiently large.

\section{Modified formulations to remove the parameter $\lambda$} \label{sec:lift}
In the following we introduce a new stabilization bilinear form based on a lifting operator and reestablish \eqref{eq:coerc}-\eqref{eq:consist} with this new parameter-free stabilization term and $\lambda=1$ in the penalty term.
For the formulations in the Sections \ref{sec:fittedpoisson}-\ref{sec:nuif} the bilinear forms $a(\cdot,\cdot)$, $N^c(\cdot,\cdot)$ and $N^s(\cdot,\cdot)$ can easily be decomposed into element contributions which are indicated by an index $T$, e.g. $a_T(\cdot,\cdot)$, so that $a(\cdot,\cdot) = \sum_{T\in\mathcal{T}_h} a_T(\cdot,\cdot)$. 
We introduce and analyse the modified formulation for the fitted Poisson case and afterwards translate it to the unfitted interface problem.
\subsection{The modified formulation for the fitted Poisson problem}
\subsubsection*{The discrete local lifting operator $\mathcal{L}(\cdot)$}
On elements which are located at the boundary we introduce an element-wise lifting into the space of polynomials which are orthogonal to constants, $\mathcal{P}_0^k(T) := \mathcal{P}^k(T) \setminus \mathcal{P}^0(T)$. The lifting is defined element-wise as 
\begin{equation} \label{eq:locallift}
\mathcal{L}_T: H^1(T) \rightarrow \mathcal{P}_0^k(T), \qquad a_T(\mathcal{L}_T(u),v) = N^c_T(v,u) \quad \forall v \in \mathcal{P}^k_0(T)
\end{equation}
We note that $\mathcal{L}_T$ is an element-local operator and that \eqref{eq:locallift} also holds for $v \in \mathcal{P}^k(T)$ as $a_T(\cdot,w)=N_T^c(w,\cdot)=0$ for $w \in \mathcal{P}^0(T)$. Further, by construction there holds $\mathcal{L}_T(u) \perp \mathcal{P}^0(T)$ and $\mathcal{L}_T(u) = 0$ if $u|_{\partial \Omega \cap T}=0$. 
We redefine $a(u,v) := \sum_{T\in \mathcal{T}_h} a_T(u,v)$ with $a_T(u,v) = (\nabla u, \nabla v)_{T}$, i.e. in a broken sense so that $a(\cdot,\cdot)$ is also well defined for only element-wise $H^1(T)$ functions. 
On elements which are not located at the boundary we set $\mathcal{L}_T(\cdot)=0$ and define the global lifting
$$
\mathcal{L}: H^1(\mathcal{T}_h) \rightarrow \bigoplus_{T\in\mathcal{T}_h} \mathcal{P}_0^k(T) \quad\text{ with }\quad \mathcal{L}(u)|_T := \mathcal{L}_T(u).
$$
Note that in general $\mathcal{L}(u)$ is discontinuous across element interfaces.
With the definition of the lifting we have $a(\mathcal{L}(u),v) = N^c(v,u)$ for all $v \in V_h^{\text{disc}} := \bigoplus_{T\in\mathcal{T}_h} \mathcal{P}^k(T)$ and hence especially for $v \in V_h \subset V_h^{\text{disc}}$.

\subsubsection*{Analysis of the modified Nitsche formulation for fitted boundary value problems}
The lifting allows to characterize the nonsymmetric term $N^c(\cdot,\cdot)$ in terms of the symmetric bilinear form $a(\cdot,\cdot)$ and the lifting operator applied on one of the arguments. This characterization will help to control $N^c(\cdot,\cdot)$ in terms of $a(\cdot,\cdot)$ and a new symmetric bilinear form involving the lifting. 
We exploit this in the following central result.
\begin{lemma} \label{lem:pos}
The bilinear form 
\begin{equation}\label{Ann}
A^{nn}(u,v) := a(u,v) + N^c(u,v) + N^c(v,u)+ \underbrace{2 a(\mathcal{L}(u),\mathcal{L}(v))}_{=:N_{nn}^s(u,v)}, \quad u,v \in V_{\text{reg}} = H^2(\Omega)
\end{equation}
is non-negative in $V_h$ and there holds
$
 A^{nn}(u_h,u_h) \geq \frac12 a(u_h,u_h) \text{ for all }u_h \in V_h.
$
\end{lemma}
\begin{proof}
Consider $u_h \in V_h$. Then with Young's inqualitywe have
$$
 2 N^c(u_h,u_h) = 2 a(\mathcal{L}(u_h),u_h) \leq 2 a(\mathcal{L}(u_h),\mathcal{L}(u_h)) + \frac12 a(u_h,u_h) = N_{nn}^s(u_h,u_h) + \frac12 a(u_h,u_h)
$$
from which we can deduce the claim.
\end{proof}
\noindent
Note that $N_{nn}^{s}(u,v)$ is well defined on $H^1(\Omega)$.
The bilinear form $A^{nn}(u,v)$ only offers control on the semi-norm $\sqrt{a(u,u)}$. 
In order to have control of a reasonable norm, i.e. in order to have control on the boundary values, we add the bilinear form $N^s_1(\cdot,\cdot)$. Note however that we can choose $\lambda=1$ (or any other positive number) as the only purpose of this part is to add control on the boundary integral. 
Due to the lifting part in $a(\mathcal{L}(\cdot),\mathcal{L}(\cdot))$ the bilinear form $N^s_{\lambda}(\cdot,\cdot)$ is no longer needed to control the non-symmetric terms $N^c(\cdot,\cdot)$. 
With 
$$
A^\ast(u,v) := A(u,v) \text{ as in \eqref{eq:nitsche} with } N^s(\cdot,\cdot) = 2 a(\mathcal{L}(\cdot),\mathcal{L}(\cdot)) + N^s_1(\cdot,\cdot)
$$
we can directly conclude the following corollary from which \eqref{eq:coerc} follows.
\begin{corollary}\label{cor:coerc}
The bilinear form $A^\ast(\cdot,\cdot)$ is well-defined on $V_{\text{reg}} + V_h$ and coercive on $V_h$ w.r.t. the norm $\vertiii{v}_A$ with coercivity constant $\frac12$.
\end{corollary}
\noindent
Consistency in the form of \eqref{eq:consist} immediately follows from the fact that $\mathcal{L}(u) = 0$ for $u|_{\partial \Omega}=0$. 
One might get the impression that the trace inverse inequality constant $c_{\text{tr}}$ does not play a role in the formulation with $A^\ast(\cdot,\cdot)$. That this is not true becomes obvious if we investigate the continuity of the bilinear form $a(\mathcal{L}(\cdot),\mathcal{L}(\cdot))$.
We have 
\begin{equation}\label{eq:contest}
\begin{split}
& a_T(\mathcal{L}(u),\mathcal{L}(u))  = N^c_T(\mathcal{L}(u),u) \leq \Vert \nabla \mathcal{L}(u) \Vert_{\Gamma\cap T} \Vert u \Vert_{\Gamma\cap T} \leq c_{\text{tr}}^{\frac12} h^{-\frac12}
\Vert \nabla \mathcal{L}(u) \Vert_{T} \Vert u \Vert_{\Gamma \cap T} \\
& \quad \leq \frac12 a_T(\mathcal{L}(u),\mathcal{L}(u)) + \frac12 \frac{c_{\text{tr}}}{h} \Vert u \Vert_{\Gamma\cap T}^2 
\quad \Longrightarrow \quad 
a(\mathcal{L}(u),\mathcal{L}(u)) \leq \frac{c_{\text{tr}}}{h} \Vert u \Vert_{\Gamma}^2~~~\forall~ u \in H^1(\Omega)
\end{split}
\end{equation}
and hence we have \eqref{eq:normeq} and \eqref{eq:cont} with constants depending on $c_{\text{tr}}$.
We see that the constant from the trace inverse estimate is still very important and enters the error analysis. However, we do not need to estimate the constant $c_{\text{tr}}$ explicitly and hence get rid of the penalty parameter $\lambda$ and the ``sufficiently large'' condition.

\subsection{The modified formulation for the unfitted interface problem} \label{sec:moduip}
The idea from the previous section can be carried over to unfitted interface problems. For this case the idea has already been introduced in \cite[Section 2.2.3.2]{lehrenfeld2015diss}. We define a lifting operator similar to \eqref{eq:locallift}, but use the space of piecewise polynomials which are orthogonal to constants on each sub-element, $\mathcal{P}_0^k(T_{1,2}) := \{ v \in L^2(T), v|_{T_1} \in \mathcal{P}_0^k(T_i),~i=1,2\}$. On cut elements we set
\begin{equation} \label{eq:locallift2}
\mathcal{L}_T: H^1(T_1\cup T_2) \rightarrow \mathcal{P}_0^k(T_{1,2}), \text{ s.t. } \quad a_T(\mathcal{L}_T(u),v) = N^c_T(v,u) \quad \forall v \in \mathcal{P}^k_0(T_{1,2}),
\end{equation}
on uncut elements we set $\mathcal{L}_T(\cdot)=0$. Overall we again have $a(\mathcal{L}(u),v) = N^c(u,v) ~~\forall v \in V_h$. With this property Lemma \ref{lem:pos}, Corollary \ref{cor:coerc} and \eqref{eq:contest} can directly be carried over and we have \eqref{eq:coerc}, \eqref{eq:normeq} and \eqref{eq:cont}. For \eqref{eq:consist} we note that $\mathcal{L}(u) = 0$ if $[\![ u ]\!]=0$.

\section{Implementational aspects}\label{sec:impl}
In this section we want to explain that the implementation of the bilinear form $2 a(\mathcal{L}(\cdot),\mathcal{L}(\cdot))$ is fairly simple. We explain the procedure for the case of the fitted Poisson problem. It however easily translates to the unfitted interface problems (or similar problems).
To implement the element-local lifting $w_h = \mathcal{L}_T(u_h)$ of a local finite element function $u_h$ and an element $T$ that is located at the boundary, we solve for $w_h$ with $u_h,w_h \in \mathcal{P}^k(T)$, such that
\begin{equation} \label{eq:local}
  a_T(w_h,v_h) +  k_T(w_h,v_h) = N^c_T(v_h,u_h), \quad \forall v_h \in \mathcal{P}^k(T)
\end{equation}  
with the bilinear form 
$$
  k_T(w_h,v_h) := h^{-(d+2)} ~ (w_h,1)_{T} ~ (v_h,1)_{T}
$$
which is taylored to eliminate the kernel $\{ u_h|_{T} = const \}$. By testing with a constant function one easily sees that if $w_h$ solves \eqref{eq:local}, we have $k_T(w_h,v_h) = 0$ for all $v_h \in \mathcal{P}_0^k(T)$.
Instead of adding $k_T(\cdot,\cdot)$ any other approach which makes \eqref{eq:local} uniquely solvable 
and ensures $a_T(w_h,v_h) = N^c_T(v_h,u_h)$ for all $v_h \in \mathcal{P}^k(T)$ is possible.
Note that the local element matrices $\mathbf{A}$ and $\mathbf{N}_c$ corresponding to the bilinear form $a_T(\cdot,\cdot)$ and $N^c_T(\cdot,\cdot)$ have to be computed anyway. 
The only new component, the rank 1 element matrix $\textbf{K}$ corresponding to $k_T(\cdot,\cdot)$ is easily obtained by scaling of a corresponding element matrix with respect to a reference element.
We thus get the coefficients $\mathbf{w}$ of the local lifting ($\mathbf{w}_i = \mathcal{L}(\varphi_i)$) as $
    \mathbf{w} =  \mathbf{L} \cdot \mathbf{u}$ with $\mathbf{L} = (\mathbf{A}+\mathbf{K})^{-1} \mathbf{N}_c^T$.
The overall element contribution to the bilinear form $A^\ast(\cdot,\cdot)$ in matrix notation is 
$$
 \mathbf{A} +  \mathbf{N}_c +  \mathbf{N}_c^T +  2 \cdot \mathbf{L}^T\mathbf{A}\mathbf{L} +  \mathbf{N}_{s,1} 
$$
where $\mathbf{N}_{s,1}$ is the element matrix corresponding to $ N^s_1(\cdot,\cdot)$. 
The element matrix corresponding to the original formulation, cf. Section \ref{sec:fittedpoisson}, takes the form
$
 \mathbf{A} +  \mathbf{N}_c +  \mathbf{N}_c^T +  \lambda \mathbf{N}_{s,1} 
$. We note that the only additional effort of the new formulation is in the computation of the local matrices $\mathbf{K}$, $\mathbf{L}$ and $2 \cdot \mathbf{L}^T\mathbf{A}\mathbf{L}$.  

\section{Numerical example}
We consider a simple numerical example to compare the original and the modified method. 
The impact of a penalty parameter that has been chosen too large is small for the fitted case as diagonal preconditioning gives robustness of the linear system with respect to $\lambda$. As this is not true for the unfitted interface problem, cf. \cite{reuskenlehrenfeld14}, we consider the discretization of an unfitted interface problem and take an example from \cite[Section 4.1]{lehrenfeld2016analysis}. On the domain $\Omega=[2.01,2.01]^2$ we consider the interface $\Gamma = \{\Vert x \Vert_4 = 1\}$ which is approximated with a piecewise planar reconstruction $\Gamma_h$. We choose the diffusion coefficients $(\alpha_1,\alpha_2)=(1,2)$ so that the solution has a kink across $\Gamma$. 
For the approximation of the solution we consider the finite element formulation as in Section \ref{sec:nuif} with $k=1$, i.e. piecewise linear functions and the modified formulation as in Section \ref{sec:moduip}. The right hand side and the Dirichlet data are chosen such that the solution is given as 
\begin{equation*}
 u(x) = \left\{ \begin{array}{rc} 1 + \frac{\pi}{2} - \sqrt{2} \cdot \cos(\frac{\pi}{4} \Vert x \Vert_4^4), & \ x \in \Omega_1, \\
\frac{\pi}{2} \Vert x \Vert_4\hphantom{)}, & \ x \in \Omega_2.  \end{array} \right.
\end{equation*}
We consider a uniform triangular grid obtained by dividing a rectangular $16\times 16$ mesh into triangles which results in $512$ unknowns. Note that this implies a very shape regular mesh. 
We computed the spectral condition number $\kappa$ of the diagonally preconditioned system matrix for different values of $\lambda$. 
\begin{table}[h!]
\small
\begin{center}
\begin{tabular}{cccccccccccccccc}
\toprule
$\lambda$    &
1 & 2 & 4 & 8 & 16 & 32 & 64 & 128 & 256 & 512 & 1024 & 2048 & 4096 & 8192
\\
\midrule
$\kappa$     &
- & - & - & - &   86.3  &   81.6  &   79.2  &   83.2  &   88.0  &   91.3  &   116.0  &   221.4  &   427.4  &   830.3 
\\
\bottomrule
\end{tabular}
\end{center}
\vspace*{-0.5cm}
\caption{Dependency of the conditioning of the Nitsche formulation on the penalty parameter $\lambda$.}
\label{tab:lam}
\vspace*{-0.5cm}
\end{table}
The results are shown in Table \ref{tab:lam}. 
We observe that the standard discretization with $N^s(\cdot,\cdot)=N_\lambda^s(\cdot,\cdot)$ is only stable for values $\lambda > 8$. Within a range of one order of magnitude the condition number seems to be essentially independent of $\lambda$. However, for $\lambda \geq 256$ the condition number increases linearly with $\lambda$. 
For the modified (parameter-free) formulation with $N^s(\cdot,\cdot) = N^{nn}(\cdot,\cdot) + N_1^s(\cdot,\cdot)$ we obtain a condition number of $86.9$, i.e. a condition number which is close to an optimal choice of $\lambda$.

Independent of $\lambda$ (for $\lambda \geq 16$) we further observed optimal order of convergence for the $L^2(\Omega)$ and the $H^1(\Omega_1 \cup \Omega_2)$ norm of the error. This holds also true for the modified formulation. Further, we observed that the error seems to be essentially independent of the choice of $\lambda$. This does not hold for the $L^2(\Gamma)$ norm of the jump which decreases linearly with increasing $\lambda$. 

\section{Conclusion}
We presented an approach to eliminate the penalty parameter from symmetric Nitsche formulations. 
Modified Nitsche formulations have been proposed which are based on well-known formulations from the literature. The new formulations inherit the good properties while the stability of the formulations no longer depends on a stabilization parameter that has to be chosen sufficiently large. We explained the procedure for two important applications of Nitsche's method and demonstrated its potential on a simple numerical example.

\bibliographystyle{ieeetr}
\bibliography{literature}

\end{document}